\documentclass[12pt]{amsart}

\usepackage{amsthm}
\usepackage{amsmath}
\usepackage{amsfonts}
\usepackage{amssymb}

\newtheorem{thm}{Theorem}[section]
\newtheorem{cor}[thm]{Corollary}
\newtheorem{lem}[thm]{Lemma}

\theoremstyle{definition}

\theoremstyle{remark}

\newtheorem{qu}[thm]{Question}

\begin{document}

\title{The Automorphism Group of the Reduced Complete-Empty $X-$Join of Graphs}

\author{ Adel Tadayyonfar$^\star$ and Ali Reza Ashrafi$^\sharp$}

\thanks{$^\star$Corresponding author (Email: adeltadayyonfar@iauln.ac.ir).\\
$^\sharp$ (Email: ashrafi@kashanu.ac.ir).}

\address{{\bf Adel Tadayyonfar}, Lenjan Branch, Islamic Azad University, Sedeh$-$Lenjan, Zarin $-$Shahr, 84741$-$68333, Isfahan, I. R. Iran}
\address{{\bf Ali Reza Ashrafi}, Department of Pure Mathematics, Faculty of Mathematical Sciences, University of Kashan, Kashan 87317$-$53153, I. R. Iran}

\dedicatory{}

\begin{abstract}
Suppose $X$ is a simple graph. The $X-$join $\Gamma$ of a set of
complete or empty graphs $\{X_x \}_{x \in V(X)}$ is a simple graph with the following vertex and edge sets:
\begin{eqnarray*}
V(\Gamma) &=& \{(x,y) \ | \ x \in V(X) \ \& \ y \in
V(X_x) \},\\ E(\Gamma) &=& \{
(x,y)(x^\prime,y^\prime) \ | \ xx^\prime \in E(X) \ or \ else \
x = x^\prime \ \& \ yy^\prime \in E(X_x)\}.
\end{eqnarray*}
The $X-$join
graph $\Gamma$ is called reduced if for vertices $x, y \in
V(X)$, $x \ne y$, $N_X(x) \setminus \{ y\} = N_X(y) \setminus \{ x\}$ implies
that $(i)$ if $xy \not\in E(X)$ then the graphs $X_x$
or $X_y$ are non-empty; $(ii)$ if $xy \in E(X)$ then
$X_x$ or $X_y$ are not complete graphs.

In this paper, we want to explore how the graph theoretical properties of  $X-$join of graphs effect on its automorphism group.
Among other results we compute the automorphism group of reduced complete-empty $X-$join of graphs.

\vskip 3mm

\noindent{\bf Keywords:} $X-$join of graphs, reduced $X-$join of graphs, automorphism group.

\vskip 3mm

\noindent{\it 2010 AMS Subject Classification Number:} Primary 20B25; Secondary 05C50.
\end{abstract}

\maketitle


\section{Introduction}

Throughout this paper all graphs are assumed to be simple and undirected. Our notations are standard and taken mainly from \cite{10,11}.
Suppose $X$ is such a graph. Sabidussi
\cite[p. 396]{5},
has defined the $X-$join of a set of graphs $\{X_x \}_{x \in V(X)}$
as the graph $\Gamma$ with vertex and edge sets
\begin{eqnarray*}
V(\Gamma) &=& \{(x,y) \ | \ x \in V(X) \ \& \ y \in V(X_x) \},\\
E(\Gamma) &=& \{ (x,y)(x^\prime,y^\prime) \ | \ xx^\prime \in E(X) \ or \  else \  x  =  x^\prime \ \& \  yy^\prime \in E(X_x)\}.
\end{eqnarray*}
This graph is obtained by replacing each vertex $x \in V(X)$
by the graph $X_x$ and inserting either all or none
of the possible edges between vertices of $X_x$ and
$X_y$ depending on whether or not $x$ and $y$ are
joined by an edge in $X$. In this paper, $X$ is assumed to be connected and
the $X-$join of complete or empty graphs $X_x$, $x \in V(X)$,
is denoted by $(\biguplus_{x \in V(X)} X_x)_X$. It is clear that when $X = K_2$, the $X-$join of
graphs $X_1$ and $X_2$ is the ordinary join and if $X =
P_n$, $n \geq 2$, then the $X-$join of graphs $X_1, \cdots,
X_{n+1}$ is the sequential join of these graphs.

Suppose $\Delta$ is a graph and $N_\Delta(x)$ denotes the set of
all neighbors of $x$ in $\Delta$. Following Habib and Maurer
\cite{55}, a subset $A$ of $V(\Delta)$ is externally
related in $\Delta$, if $N_\Delta(x) \setminus A = N_\Delta(y) \setminus A$,
for all $x, y \in A$. Obviously, $\emptyset$, $\{ x\}$, $x \in
V(\Delta)$, and $V(\Delta)$ are externally related. The complete-empty
$X-$join graph $\Gamma = (\biguplus_{x \in V(X)}X_x)_X$ is called
reduced if for vertices $x, y \in V(X)$, $x \ne y$, $N_X(x) \setminus
\{ y\} = N_X(y) \setminus \{ x\}$ implies that $(i)$ if $xy \not\in
E(X)$ then at least one of $X_x$ and $X_y$ is not an
empty graph; $(ii)$ if $xy \in E(X)$ then at least one of
$X_x$ and $X_y$ is not a complete graph.

Suppose $\Gamma_1$ and $\Gamma_2$ are graphs with disjoint vertex sets. The lexicographic product of $\Gamma_1$ and $\Gamma_2$ is another graph $\Gamma_1 o \Gamma_2$ with vertex set $V(\Gamma_1) \times V(\Gamma_2)$ and two vertices $(x_1,y_1)$ and $(x_2,y_2)$ are adjacent if and only if $x_1x_2 \in E(\Gamma_1)$ or $x_1 = x_2$ and $y_1y_2 \in E(\Gamma)$. Note that the lexicographic product is not commutative. If $\Gamma$ is the $X-$join of  graphs $\{X_x \}_{x \in V(X)}$ and $X_x \cong X_y$, for each $x, y \in V(X)$, then $\Gamma \cong X o X_x$, for some $x \in V(X)$.

\begin{lem}
Suppose $x , y$ are vertices of a simple graph $\Delta$. Then
$N_{\Delta}(x) \setminus \{y\} = N_{\Delta}(y) \setminus \{x\}$
if and only if $\{x , y\}$ is externally related.
\end{lem}

\begin{proof}
It is easy to see that
$N_{\Delta}(x) \setminus \{x , y\} \subseteq N_{\Delta}(x) \setminus \{y\}$
and $N_{\Delta}(y) \setminus \{x , y\} \subseteq N_{\Delta}(y) \setminus \{x\}$.
Since $\Delta$ is simple,
$N_\Delta(x) \setminus \{ y\} \subseteq N_\Delta(x) \setminus \{ x,y\}$ and
$N_\Delta(y) \setminus \{ x\} \subseteq N_\Delta(y) \setminus \{ x,y\}$ proving the lemma.
\end{proof}

The complete and empty graphs on a non-empty set $B$ are denoted
by $K_B$ and $\Phi_B$, respectively. In the case that $|B| = n$
we use the notation $K_n$ as $K_B$ and $\Phi_n$ as $\Phi_B$. The
complete bipartite graphs $K_{m,n}$ can be constructed as
$K_{m,n} = \Phi_m + \Phi_n$. If $\Sigma$ and $\Delta$ are graphs
with $V(\Sigma) \subseteq V(\Delta)$ and $E(\Sigma) \subseteq
E(\Delta)$, then we say $\Sigma$ is a subgraph of $\Delta$ and
write $\Sigma \leq \Delta$. If $T \subseteq V(\Delta)$ then the
induced subgraph $\Delta[T]$ is a subgraph with $V(\Delta[T]) =
T$ and $E(\Delta[T]) = \{ e=uv \in E(\Delta) \ | \ \{ u, v\}
\subseteq T\}$. If $C$ and $D$ are subsets of $V(\Delta)$ and all
elements of $C$ are adjacent to all elements of $D$, then we write
$C \sim D$. If there is no element in $C$ to be adjacent with an
element of $D$, then we use the notation $C \nsim D$. If $\Gamma$ is a graph and $\mathcal{P}$ is a partition of $V(\Gamma)$ then the quotient graph $\frac{\Gamma}{\mathcal{P}}$ has the vertex set $\mathcal{P}$ and two partitions $V_1$ and $V_2$ are adjacent if there are $v_1 \in V_1$ and $v_2 \in V_2$ such that $v_1v_2 \in E(\Gamma)$. Our other
notations are standard and can be taken from the standard book on
graph theory.

The aim of this paper is to compute the automorphism group of the reduced
complete-empty $X-$join graphs. To do this, we assume that
$\Gamma$ is such a graph.
Choose $X_x$, $x \in V(X)$, to be the subgraph corresponding to the vertex
$x$ in $X$. Define $x \thickapprox y$, if and only if
$X_x \cong X_y$, where $x, y \in V(X)$. It is easy to
see that $\thickapprox$ is an equivalence relation.
Moreover, we assume that $T_x$ denotes the equivalence class of
$x$ under $\thickapprox$ and $W$ is a set of representatives for
equivalence relation  $\thickapprox$. Define
$\mathcal{A}(X) = \{f \in Aut(X) \ | \ \forall x \in W , f(T_x) = T_x\}$
as a subgroup of $Aut(X)$. Our main result is:

\begin{thm}\label{main teorem}
Suppose $\Gamma$ is a reduced complete-empty $X-$join of graphs $X_x$, $x \in V(X)$. Then,
$$Aut(\Gamma) \cong \left(\prod_{x \in V(X)}Sym(V(X_x))\right) \rtimes \mathcal{A}(X).$$
\end{thm}


\section{Proof of the Main Theorem}

The aim of this section is to prove the main theorem of this paper.

\begin{lem}\label{lemma of aut.}

Suppose $\Gamma$ is a reduced complete-empty
$X-$join and $X_x$ is a graph corresponding to the vertex $x$ of $X$.
If $\sigma \in Aut(\Gamma)$ satisfies
this condition that for each $x \in V(X)$, there exists $y \in V(X)$,
such that $\sigma(X_x) = X_y$ then the function $f: V(X) \longrightarrow V(X)$
given by $f(x) = y$ is an automorphism of $X$.

\end{lem}

\begin{proof}
We assume that $t_x \in V(X_x)$, for each $x \in V(X)$. Then we have:
\begin{eqnarray*}
xx^{\prime} \in E(X) &\Leftrightarrow& t_xt_{x^{\prime}} \in E(\Gamma)\\
&\Leftrightarrow& \sigma(t_x)\sigma(t_{x^{\prime}}) \in E(\Gamma)\\
&\Leftrightarrow& yy^{\prime} \in E(X)\\
&\Leftrightarrow& f(x)f(x^{\prime}) \in E(X),
\end{eqnarray*}
where $y^{\prime} = f(x^{\prime})$, proving the lemma.
\end{proof}

\begin{thm}

\label{p1} Suppose $\Gamma$ is a reduced complete-empty
$X-$join and $X_x$ is a graph corresponding to the vertex $x$ of $X$. If for each $y \in V(X)$,
$X_x \cong X_y$, then $Aut(\Gamma) \cong
Sym(V(Y)) \wr_{V(X)} Aut(X)$, where $Y \cong X_x$.

\end{thm}

\begin{proof}
If $|V(X)| = 1, 2$  or $|V(X_x)| = 1$ then the proof will be clear.
Hence we can assume that $|V(X)| \geq 3$ and $|V(X_x)| \geq 2$.
Since for each $x \in V(X)$, $X_x$'s are isomorphism and they are complete or empty,
$Aut(X_x) \cong Sym(V(X_x))$. Define:
$$U = \{\sigma \in Sym(\Gamma) \ | \ \exists f \in Aut(X) , \forall x \in V(X) , \sigma(X_x) = X_{f(x)}\}.$$
We first prove that $U = Aut(\Gamma)$. To prove
$U \subseteq Aut(\Gamma)$, we assume that
$\sigma \in U$ and $a , b \in V(\Gamma)$. If there exists $x \in V(X)$,
such that $a , b \in V(X_x)$ then there exists
$f \in Aut(X)$, such that $\sigma(a), \sigma(b) \in \sigma(V(X_x)) =  V(X_{f(x)})$. This
shows that $ab \in E(\Gamma)$ if and only if $\sigma(a)\sigma(b) \in E(\Gamma)$.
We now assume that there are $x , y \in V(X)$
such that $x \ne y$, $a \in V(X_x)$ and $b \in V(X_y)$.
Then obviously there is an automorphism $f \in Aut(X)$
such that $\sigma(a) \in \sigma(X_x) = X_{f(x)}$,
$\sigma(b) \in \sigma(X_y) = X_{f(y)}$. Therefore,
\begin{eqnarray*}
ab \in E(\Gamma) &\Leftrightarrow& xy \in E(X)\\
&\Leftrightarrow& f(x)f(y) \in E(X)\\
&\Leftrightarrow& X_{f(x)} \sim X_{f(y)}\\
&\Leftrightarrow& \sigma(a)\sigma(b) \in E(\Gamma).
\end{eqnarray*}
This proves that $U \subseteq Aut(\Gamma)$. To prove $Aut(\Gamma) \subseteq U$
we assume that $\theta \in Aut(\Gamma)$, $x \in V(X)$ and $a \in V(X_x)$.
Since $|V(X_x)| \geq 2$, there exists a vertex $b \in V(X_x)$ such that
$b \neq a$. There are two cases for $\theta(a)$ as follows:

\begin{enumerate}

\item

$\theta(a) \in V(X_x)$. We should prove that $\theta(b) \in V(X_x)$.
Let us assume that, on the contrary, there exists
$y \in V(X)$ such that $x \neq y$ and $\theta(b) \in V(X_y)$. If all $X_x$ are complete then
$$ab \in E(X_x) \Rightarrow ab \in E(\Gamma) \Rightarrow \theta(a)\theta(b) \in E(\Gamma) \Rightarrow X_x \sim X_y.$$
If $|V(X)| = 2$ then clearly $N_X(x) \setminus \{y\} = N_X(y) \setminus \{x\} = \emptyset$
contradict by reducibility of $\Gamma$ over $X$. Suppose $|V(X)| \geq 3$, $z \in
N_X(x) \setminus \{y\}$ and $c \in X_z$. Then,
\begin{eqnarray*}
ac \in E(\Gamma) &\Leftrightarrow& \theta(a)c \in E(\Gamma) \ \ \ \ \ \ \ \ \ \ \ \ \ \ \ \ (\theta(a) \in V(X_x))\\
&\Leftrightarrow& a\theta^{-1}(c) \in E(\Gamma)\\
&\Leftrightarrow& b\theta^{-1}(c) \in E(\Gamma) \ \ \ \ \ \ \ \ \ \ \ \ \ \ (b \in V(X_x))\\
&\Leftrightarrow& \theta(b)c \in E(\Gamma)\\
&\Leftrightarrow& z \in N_X(y) \setminus \{x\} \ \ \ \ \ \ \ \ \ \ \ \ \ (\theta(b) \in V(X_y)).
\end{eqnarray*}
Hence $N_X(x) \setminus \{y\} = N_X(y) \setminus \{x\}$ which is impossible.
If all $X_x$'s are empty then there is no edge in
$X_x$ connecting $a$ and $b$ and so $\theta(a)$ is not adjacent to
$\theta(b)$. So, there is no an edge connecting
$X_x$ and $X_y$ in $\Gamma$. Again, we can prove that
$N_X(x) \setminus \{y\} = N_X(y) \setminus \{x\}$, a contradiction.
So, $\theta(b) \in V(X_x)$.

\item

$\theta(a) \not\in V(X_x)$. In this case there exists $y \in V(X)$ such that
$\theta(a) \in V(X_y)$. We prove that $\theta(b) \in V(X_y)$.
By the contrary, we assume that there exists
$z \in V(X)$ such that $y \neq z$ and $\theta(b) \in V(X_z)$.
If $X_x$'s are complete then we have $ab \in E(X_x)$ which
proves that $\theta(a)\theta(b) \in E(\Gamma)$. So, $X_y \sim X_z$.
If $N_X(y) \setminus \{z\} = N_X(z) \setminus \{y\} = \emptyset$, then this is
contradict by reducibility of $\Gamma$ over $X$.
If $t \in N_X(z) \setminus \{y\}$ and $c \in X_t$ then
\begin{eqnarray*}
tz \in E(X) &\Leftrightarrow& c\theta(b) \in E(\Gamma)\\
&\Leftrightarrow& \theta^{-1}(c)b \in E(\Gamma)\\
&\Leftrightarrow& \theta^{-1}(c)a \in E(\Gamma)\\
&\Leftrightarrow& c\theta(a) \in E(\Gamma)\\
&\Leftrightarrow& ty \in E(X).
\end{eqnarray*}
This implies that $N_X(y) \setminus \{z\} = N_X(z) \setminus \{y\}$, contradict by the fact
that $\Gamma$ is a reduced complete-empty $X-$join.
Now, we assume that all $X_x$'s are empty. If  $|V(X)| = 2$ then
$\Gamma$ is isomorphic to a complete bipartite graph and so
$\theta(b) \in V(X_y)$, which is impossible.
If $|V(X)| \geq 3$ then similar to the
previous case, $N_X(y) \setminus \{z\} = N_X(z) \setminus \{y\}$ which is another
contradiction. Thus, $\theta(b) \in V(X_y)$.

\end{enumerate}

By above discussion, for each $x \in V(X)$, there exists a unique
$y_x \in V(X)$ such that $\theta(X_x) = X_{y_x}$.
Now, By defining $g(x) = y_x$ and Lemma
\ref{lemma of aut.}, one can see that $\theta(X_x) = X_{g(x)}$.
Therefore, $Aut(\Gamma) \subseteq U$ and
$$Aut(\Gamma) = \{\sigma \in Sym(\Gamma) \ | \ \exists f \in Aut(X) , \forall x \in V(X) , \sigma(X_x) = X_{f(x)}\}.$$
There exists $\varphi : Aut(\Gamma)
\rightarrow Aut(X)$ given by $\varphi(\sigma) =
f_{\sigma}$, $\sigma \in Aut(\Gamma)$, is an onto
homomorphism such that for each $x \in V(X)$, $\sigma(X_x) = X_{f_{\sigma}(x)}$.
Therefore,
\begin{eqnarray*}
Ker(\varphi) &=& \{\sigma \in Aut(\Gamma) \ | \ \varphi(\sigma) = e_{Aut(X)}\}\\
&=& \{\sigma \in Aut(\Gamma) \ | \ f_{\sigma} = e_{Aut(X)}\}\\
&=& \{\sigma \in Aut(\Gamma) \ | \ \forall x \in V(X) \ , \ \sigma(X_x) = X_x\}\\
&=& \{\sigma \in Aut(\Gamma) \ | \ \sigma = \Pi_{x \in V(X)} \sigma_x \ s.t. \ \forall x \in V(X) \ ,
\ \sigma_x \in Sym(V(X_x))\}\\
&=& \prod_{x \in V(X)} Sym(V(X_x)) \cong \prod_{x \in V(X)} Sym(Y),
\end{eqnarray*}
By previous notations, $\Gamma \cong XoY$ and $\sigma_x \in Sym(V(Y))$.
Without loss of generality, one can consider that $\sigma_x \in Sym(\{x\} \times V(Y))$. Set
$$B = \{\sigma_f \in Aut(XoY) \ | \ f \in Aut(X) \ , \ \forall (x , y) \in V(XoY) , \sigma_f((x , y)) = (f(x) , y)\}.$$
It is clear that $B \cong Aut(X)$ and
$$B \cap Ker(\varphi) = \{\sigma_f \in B \ | \ \forall x \in V(X) \ , \ f(x) = x\} = \{e_{Aut(\Gamma)}\}.$$
By definition of $U$ and for each $(x , y) \in V(XoY)$, we have:
$$\sigma((x , y)) = (f(x) , \sigma_x(y)) = \sigma_f((x , \sigma_x(y))) = \sigma_f(\sigma_x((x , y))) = (\sigma_fo\sigma_x)((x , y)).$$
Therefore $\sigma \in BKer(\varphi)$ and so $Aut(\Gamma) =
BKer(\varphi)$. Since $Ker(\varphi) \unlhd Aut(\Gamma)$,
\begin{eqnarray*}
Aut(\Gamma) &\cong& Ker(\varphi) \rtimes B\\ &\cong&
\prod_{x \in V(X)} Sym(Y) \rtimes Aut(A)\\
&\cong& Sym(Y) { \wr}_{V(X)} Aut(X).
\end{eqnarray*}
This completes the proof.
\end{proof}

It is merit to mention here that Theorem \ref{p1} can be proved by \cite[Theorem 3.1]{7}, but our proof is independent from some technical concepts like  natural isomorphism, collapsed graph, section graph, $X-$subjoin, smorphism and inverting $X-$point. On the other hand,  three parts of the proof of Theorem \ref{p1} are needed to complete the proof of main theorem.

\begin{thm}\label{p2}

Suppose $\Gamma$ is a reduced complete-empty $X-$join and $X_x$
is a graph corresponding to the vertex $x$ of $X$. If for each $x \neq y \in V(X)$,
$X_x \not\cong X_y$, then $Aut(\Gamma) \cong \prod_{x \in V(X)} Sym(V(X_x))$.

\end{thm}

\begin{proof}
If $|V(X)| = 1$ or $2$ then the proof is trivial.
Suppose $|V(X)| > 2$ and $\sigma \in Aut(\Gamma)$ is arbitrary.
We prove that for each $x \in V(X)$, $\sigma(X_x) = X_x$.
Since $X_x$'s are non-isomorphic, for each $x \neq y \in V(X)$,
$\sigma(X_x) \neq X_y$. Suppose that there are $x \in V(X)$
and $a \in V(X_x)$ with this property that
$\sigma(a) \not \in V(X_x)$. Hence there exists $y \in V(X)$, $y \neq x$, such that $\sigma(a) \in
V(X_y)$. We consider four separate cases as follows:

\begin{enumerate}

\item

$X_x \cong K_1$. Suppose $V(X_x) = \{a\}$. Since $X_s$'s are mutually non-isomorphic,
$|V(X_y)| \geq 2$. So, there are $b \in V(X_y)$ and
$z \in V(X)$, such that $b \neq \sigma(a)$ and
$\sigma^{-1}(b) \in V(X_z)$. If $N_X(x) \neq \{z\}$, then we assume that $t \in
N_X(x) \setminus \{z\}$ and $c \in V(X_t)$. Then,
\begin{eqnarray*}
tx \in E(X) &\Leftrightarrow& ca \in E(\Gamma) \ \ \ \ \ \ \ \ \ \ \ \ \ \ \ \ \ \ (c \in V(X_t) , a \in V(X_x))\\
&\Leftrightarrow& \sigma(c)\sigma(a) \in E(\Gamma)\\
&\Leftrightarrow& \sigma(c)b \in E(\Gamma) \ \ \ \ \ \ \ \ \ \ \ \ \ \ (b , \sigma(a) \in V(X_y))\\
&\Leftrightarrow& c\sigma^{-1}(b) \in E(\Gamma)\\
&\Leftrightarrow& tz \in E(X) \ \ \ \ \ \ \ \ \ \ \ \ \ \ \ \ \ \ (\sigma^{-1}(b) \in V(X_z)).
\end{eqnarray*}
Hence $N_X(x) \setminus \{z\} = N_X(z) \setminus \{x\}$.
If $N_X(x) = \{z\}$ then $N_X(x)$ $\setminus$ $\{z\}$  $=$ $N_X(z)$ $\setminus$ $\{x\}$ = $\emptyset$.
Since $|V(X_z)| \geq 2$, we can choose $d \neq
\sigma^{-1}(b)$ in $X_z$. Then
\begin{eqnarray*}
b \sigma(a) \in E(\Gamma) &\Leftrightarrow& \sigma^{-1}(b)a \in E(\Gamma)\\
&\Leftrightarrow& zx \in E(X) \ \ \ \ \ \ \ \ \ \ \ \ \ \ (\sigma^{-1}(b) \in V(X_z) , a \in V(X_x))\\
&\Leftrightarrow& da \in E(\Gamma) \ \ \ \ \ \ \ \ \ \ \ \ \ \ \ (d \in V(X_z))\\
&\Leftrightarrow& \sigma(d)\sigma(a) \in E(\Gamma)\\
&\Leftrightarrow& \sigma(d)b \in E(\Gamma) \ \ \ \ \ \ \ \ \ \ \ \ (b , \sigma(a) \in V(X_y))\\
&\Leftrightarrow& d\sigma^{-1}(b) \in E(\Gamma).
\end{eqnarray*}
This means that $X_y$ is complete if and only if $X_z$ is complete.
If $X_y$ is a complete graph, then $\sigma^{-1}(b)a \in E(\Gamma)$ and
so $X_x \sim X_z$. But, $X_x \cong K_1$ which
contradicts by reducibility of $\Gamma$ over $X$. Thus,
$X_y$ is empty and we have $b\sigma(a) \not \in
E(\Gamma)$. Therefore, $X_z$ is empty and $X_x \not \sim X_z$, which is impossible.
This proves that $\sigma(X_x) = X_x$.

\item

$X_x \ncong K_1$ and there exists $b \in V(X_x)$
such that $a \neq b$ and $\sigma(b) \in V(X_x)$.
$X_x$ is complete if and only if $ab \in E(\Gamma)$. It means that
$\sigma(a)\sigma(b) \in E(\Gamma)$ and so $X_x \sim X_y$. At first, suppose that
$N_X(x) \neq \{y\}$. Assume that $z \in N_X(x) \setminus \{y\}$ and $c \in V(X_z)$. Then we have:
\begin{eqnarray*}
xz \in E(X) &\Leftrightarrow& \sigma(b)c \in E(\Gamma) \ \ \ \ \ \ \ \ \ \ \ \ \ \ \ \ \ \ (c \in V(X_z) , \sigma(b) \in V(X_x))\\
&\Leftrightarrow& b\sigma^{-1}(c) \in E(\Gamma)\\
&\Leftrightarrow& a\sigma^{-1}(c) \in E(\Gamma) \ \ \ \ \ \ \ \ \ \ \ \ \ \ \ \ (a , b \in V(X_x))\\
&\Leftrightarrow& \sigma(a)c \in E(\Gamma)\\
&\Leftrightarrow& yz \in E(X) \ \ \ \ \ \ \ \ \ \ \ \ \ \ \ \ \ \ \ \ \ \ (\sigma(a) \in V(X_y)).
\end{eqnarray*}
So $N_X(x) \setminus \{y\} = N_X(y) \setminus \{x\}$. If
$N_X(x) = \{y\}$ then again we have $N_X(x) \setminus \{y\} = N_X(y) \setminus \{x\} = \emptyset$.
Since $\Gamma$ is reduced,
$X_x$ is complete, if and only if $X_y$ is an empty graph. If
$|V(X_y)| = 1$ then $X_y \cong K_1$ or $X_y \cong \Phi_1$, which leaded us to
another contradiction. Therefore, there exists
$d \in V(X_y)$, $d \neq \sigma(a)$. Hence,
\begin{eqnarray*}
xy \in E(X) &\Leftrightarrow& \sigma(b)d \in E(\Gamma) \ \ \ \ \ \ \ \ \ \ \ \ \ \ \ \ (\sigma(b) \in V(X_x) , d \in V(X_y))\\
&\Leftrightarrow& b\sigma^{-1}(d) \in E(\Gamma)\\
&\Leftrightarrow& a\sigma^{-1}(d) \in E(\Gamma) \ \ \ \ \ \ \ \ \ \ \ \ \ (a , b \in V(X_x))\\
&\Leftrightarrow& \sigma(a)d \in E(\Gamma).
\end{eqnarray*}
It means that $X_x$ is complete, if and only if $X_y$ is complete which is a contradiction.
Therefore, for each $x \in V(X)$, $\sigma(X_x) = X_x$.

\item

$X_x \ncong K_1$ and there exists $b \in V(X_x)$ such that
$b \neq a$ and $\sigma(b) \not \in V(X_x) \cup V(X_y)$.
In this case, there exists $z \in V(X_x)$
such that $z \neq x , y$ and $\sigma(b) \in
V(X_z)$. $X_x$ is complete if and only if $ab$ is an edge of $\Gamma$
and so $\sigma(a)\sigma(b) \in E(\Gamma)$. It means that $X_x$
is a complete graph if and only if $yz \in E(X)$.
By assumption $N_X(z) \neq \{y\}$. Set $t \in N_X(z) \setminus \{y\}$ and
$c \in V(X_t)$. Since
\begin{eqnarray*}
tz \in E(X) &\Leftrightarrow& c\sigma(b) \in E(\Gamma) \ \ \ \ \ \ \ \ \ \ \ \ \ \ \ \ (c \in V(X_t) , \sigma(b) \in V(X_z))\\
&\Leftrightarrow& \sigma^{-1}(c)b \in E(\Gamma)\\
&\Leftrightarrow& \sigma^{-1}(c)a \in E(\Gamma) \ \ \ \ \ \ \ \ \ \ \ \ \ (a , b \in V(X_x))\\
&\Leftrightarrow& c\sigma(a) \in E(\Gamma)\\
&\Leftrightarrow& ty \in E(X) \ \ \ \ \ \ \ \ \ \ \ \ \ \ \ \ \ \ \ \ (\sigma(a) \in V(X_y)),
\end{eqnarray*}
$N_X(y) \setminus \{z\} = N_X(z) \setminus \{y\}$.
If $N_X(z) = \{y\}$ then again $N_X(y) \setminus \{z\} = N_X(z) \setminus \{y\} = \emptyset$.
Since $\Gamma$ is reduced, exactly one of $X_y$ or
$X_z$ are empty graph. On the other hand, $X_s$'s are
mutually non-isomorphic and so none of $X_y$ or
$X_z$ have one vertex. Choose $d \in V(X_y)$
and $e \in V(X_z)$. The following discussion shows that $X_y$ is complete
if and only if $X_z$ is a complete graph which is a
contradiction:
\begin{eqnarray*}
\sigma(a)d \in E(\Gamma) &\Leftrightarrow& a\sigma^{-1}(d) \in E(\Gamma)\\
&\Leftrightarrow& b\sigma^{-1}(d) \in E(\Gamma) \ \ \ \ (a , b \in V(X_x))\\
&\Leftrightarrow& \sigma(b)d \in E(\Gamma)\\
&\Leftrightarrow& e\sigma(a) \in E(\Gamma) \ \ \ \ \ \ \ (\sigma(b) , e \in V(X_z) , d , \sigma(a) \in V(X_y))\\
&\Leftrightarrow& \sigma^{-1}(e)a \in E(\Gamma)\\
&\Leftrightarrow& \sigma^{-1}(e)b \in E(\Gamma)\\
&\Leftrightarrow& e\sigma(b) \in E(\Gamma).
\end{eqnarray*}
So, $\sigma(X_x) = X_x$.

\item

$X_x \ncong K_1$ and there exists $b \in V(X_x)$ such that $a \neq b$ and $\sigma(b) \in V(X_y)$. Since
$ab \in E(\Gamma)$ if and only if $X_x$ is complete, one can easily see that
$X_x$ is a complete graph, if and only if $X_y$ is complete. If
$|V(X_x)| = |V(X_y)| = 2$ then $X_x \cong X_y$, which is impossible.
If $|V(X_y)| = 2$ then $|V(X_x)| \geq 3$ and so there exists
$c \in V(X_x)$ such that $\sigma(c) \not \in V(X_y)$. But, this is the case $(3)$ which is a contradiction.
Thus $|V(X_y)| \geq 3$. Then there exists an element
$c \in V(X_y)$ such that  $c \neq \sigma(a) , \sigma(b)$. But there
is $z \in V(X)$, such that $\sigma^{-1}(c)
\in V(X_z)$. We assume that $N_X(x) \setminus \{z\} \neq \emptyset$,  $t \in N_X(x) \setminus \{z\}$ and $d \in
V(X_t)$. Then
\begin{eqnarray*}
tx \in E(X) &\Leftrightarrow& da \in E(\Gamma) \ \ \ \ \ \ \ \ \ \ \ \ \ \ \ \ \ \ (d \in V(X_t) , a \in V(X_x))\\
&\Leftrightarrow& \sigma(d)\sigma(a) \in E(\Gamma)\\
&\Leftrightarrow& \sigma(d)c \in E(\Gamma) \ \ \ \ \ \ \ \ \ \ \ \ \ \ \ (\sigma(a) , c \in V(X_y))\\
&\Leftrightarrow& d \sigma^{-1}(c) \in E(\Gamma)\\
&\Leftrightarrow& tz \in E(X) \ \ \ \ \ \ \ \ \ \ \ \ \ \ \ \ \ \ \ (\sigma^{-1}(c) \in V(X_z)).
\end{eqnarray*}
Therefore,
$N_X(x) \setminus \{z\} = N_X(z) \setminus \{x\}$. If $X_y$ is complete then $\sigma(a)c$
is an edge of $\Gamma$, which implies that $a\sigma^{-1}(c) \in E(\Gamma)$. Ultimately, $xz \in E(X)$.
By the similar argument, we can see that if $xz \in E(X)$, then the graph $X_y$ is complete.
Now, reducibility of $\Gamma$ over $X$ concludes that $X_x$
is complete, if and only if $X_z$ is an empty graph.
Hence, $|V(X_z)| \geq 2$ and so, there exists $e \in V(X_z)$ such that $e \neq \sigma^{-1}(c)$.
Now, we have
\begin{eqnarray*}
ab \in E(\Gamma) &\Leftrightarrow& \sigma(a)\sigma(b) \in E(\Gamma)\\
&\Leftrightarrow& \sigma(a)c \in E(\Gamma) \ \ \ \ \ \ \ \ \ (\sigma(b) , c \in V(X_y))\\
&\Leftrightarrow& a\sigma^{-1}(c) \in E(\Gamma)\\
&\Leftrightarrow& be \in E(\Gamma) \ \ \ \ \ \ \ \ \ \ \ \ \ \ (a , b \in V(X_x) , \sigma^{-1}(c) , e \in V(X_z))\\
&\Leftrightarrow& \sigma(b)\sigma(e) \in E(\Gamma)\\
&\Leftrightarrow& c\sigma(e) \in E(\Gamma)\\
&\Leftrightarrow& \sigma^{-1}(c)e \in E(\Gamma).
\end{eqnarray*}
Above argument shows that $X_x$ is complete if and only if
$X_z$ is a complete graph, which is our final contradiction. So, we have again
$\sigma(X_x) = X_x$, as desired.
\end{enumerate}

Therefore, each $\sigma, \sigma \in Aut(\Gamma)$, has a decomposition of $\sigma_x$'s, where
$\sigma_x \in Aut(X_x)$. Thus, $Aut(\Gamma)$ can be written as an inner product of
$Aut(X_x)$'s. But, for every $y$, $y \in V(X)$ and $y \neq x$,
$Aut(X_x) \cap Aut(X_y) = \{e_{Aut(\Gamma)}\}$. Since $X_x$'s are complete or empty,
$Aut(X_x) \cong Sym(V(X_x))$. Therefore,
$Aut(\Gamma) \cong \prod_{x \in V(X)} Sym(V(X_x))$.
This completes the proof.
\end{proof}

\noindent{\bf Proof of the Main Theorem}.
It is clear that if
$|V(X)| = 1$ or $|V(X_x)| = 1$, $x \in V(X)$, then the proof is trivial.
So, suppose $|V(X)| , |V(X_x)| \geq 2$. Since all $X_x$'s are complete or empty,
$Aut(X_x) \cong Sym(V(X_x))$, for each $x \in V(X)$. Define
$$C = \{\sigma \in Sym(\Gamma) \ | \ \exists f \in \mathcal{A}(X) , \forall x \in V(X) , \sigma(X_x) = X_{f(x)}\}.$$
We prove that $C = Aut(\Gamma)$. Let $\sigma \in C$ and $a , b \in V(\Gamma)$.
If there exists $x \in V(X)$ such that $a , b \in V(X_x)$, then
$\sigma(a) , \sigma(b) \in V(\sigma(X_x)) = V(X_{f(x)})$, for some
$f \in Aut(X)$ that $f(T_x) = T_x$. Since $X_x \cong X_{f(x)}$,
$ab \in V(\Gamma)$ if and only if $\sigma(a)\sigma(b) \in E(\Gamma)$.
So, $C \subseteq Aut(\Gamma)$. If there exists $x , y \in V(X)$,
such that $x \neq y$, $a \in V(X_x)$ and $b \in V(X_y)$, then by a similar argument as Theorem
\ref{p1},
again one can conclude that $C \subseteq Aut(\Gamma)$.
Conversely, suppose $\theta \in Aut(\Gamma)$, $x \in V(X)$ and $a \in V(X_x)$. We show that
there exists $h \in \mathcal{A}(X)$ such that $\theta(X_x) = X_{h(x)}$.
If $|V(X_x)| = 1$ then by a similar argument as the proof of Theorem
\ref{p2} (a),
$\theta(X_x) = X_y$ in which $y \in T_x$. Now, we proceed
the proof   by assuming that $|V(X_x)| \geq 2$.
Choose $b$, $b \neq a$, be an arbitrary vertex of $X_x$.
We have three separate cases as follows:

\begin{itemize}

\item[a.]

$\theta(a) \in V(X_x)$. In this case, if $\theta(b) \in V(X_y)$, $y \in T_x$, then
a similar argument as the proof of Theorem
\ref{p1} (a),
$x = y$. Furthermore, if $\theta(b) \in V(X_z)$, for some $z \in V(X)$ such that $z \not \in T_x$,
then a similar argument as the proof of Theorem
\ref{p2} (b),
shows that it is impossible. So $\theta(b) \in V(X_x)$.

\item[b.]

$\theta(a) \in V(X_y)$ such that $y \in T_x$. In this case, the argument of the previous case applies for $x = y$.
If $\theta(b) \in V(X_z)$ such that $z \in T_x$, then a similar argument as the proof of Theorem
\ref{p1} (b),
lead us to $y = z$. Ultimately, assume that $\theta(b) \in V(X_t)$, for some $t \in V(X)$
such that $t \not \in T_x$. It is easy to see that $X_x$ is complete if and only if $ab \in E(\Gamma)$ and also,
$yz \in E(X)$ if and only if $\theta(a)\theta(b) \in E(\Gamma)$. So, $X_x$ is a complete graph if and only if
$yz \in E(X)$. Since $X_x$ and $X_y$ are isomorphic, $X_y$ is complete if and only if
$yz \in E(X)$. If $N_X(y) \setminus \{z\} = \emptyset$ then
$N_X(y) \setminus \{z\} = N_X(z) \setminus \{y\}$.
Suppose $N_X(y) \setminus \{z\} \neq \emptyset$,
$t \in N_X(y) \setminus \{z\}$ and $c \in V(X_t)$. Then,
\begin{eqnarray*}
ty \in E(X) &\Leftrightarrow& c\theta(a) \in E(\Gamma)\\
&\Leftrightarrow& \theta^{-1}(c)a \in E(\Gamma)\\
&\Leftrightarrow& \theta^{-1}(c)b \in E(\Gamma)\\
&\Leftrightarrow& c\theta(b) \in E(\Gamma)\\
&\Leftrightarrow& tz \in E(X).
\end{eqnarray*}
By above relations and the fact that $\Gamma$ is a reduced complete-empty
$X-$join, $X_x$ and $X_y$ are complete graphs if and only if $X_z$ is empty.
On the other hand, by the following relations and the fact that $X_z$ is an empty graph if and only if $d\theta(b) \not \in E(\Gamma)$,
one can deduce that  $X_y$ is complete if and only if $yz \not \in E(X)$,
which is a contradiction:
\begin{eqnarray*}
d\theta(b) \not \in E(\Gamma) &\Leftrightarrow& \theta^{-1}(d)b \not \in E(\Gamma)\\
&\Leftrightarrow& \theta^{-1}(d)a \not \in E(\Gamma)\\
&\Leftrightarrow& d\theta(a) \not \in E(\Gamma)\\
&\Leftrightarrow& yz \not \in E(X).
\end{eqnarray*}

\item[c.]

$\theta(a) \in V(X_z)$ and $z \not \in T_x$. If $\theta(b) \in E(X)$ then
a similar argument as the proof of Theorem
\ref{p2} (b),
implies that it is impossible. Moreover, if $\theta(b) \in V(X_t)$ and $t \in T_x$,
then the previous case deduces that it is not possible, too.
By the proof of Theorem
\ref{p2} (c),
if $\theta(b) \in V(X_s)$ and $s \not \in T_x$, then we get another contradiction.
Therefore, $\theta(b) \in V(X_z)$.

\end{itemize}

We observe in three cases that for each
$x \in V(X)$, there exists $y_x \in V(X)$
such that $\theta(X_x) = X_{y_x}$.
By putting $h(x) = y_x$ and Lemma \ref{lemma of aut.},
we have $\theta(X_x) = X_{h(x)}$.
This proves $C = Aut(\Gamma)$.
One can easily see that the map
$\varphi : Aut(\Gamma) \rightarrow \mathcal{A}(X)$
given by $\varphi(\sigma) = f_{\sigma}$
is an onto group homomorphism and for every $x \in V(X)$,
$\sigma(X_x) = X_{f_{\sigma}(x)}$.
Hence,
\begin{eqnarray*}
Ker(\varphi) &=& \{\sigma \in Aut(\Gamma) \ | \ \varphi(\sigma) = e_{\mathcal{A}(X)}\}\\
&=& \{\sigma \in Aut(\Gamma) \ | \ f_{\sigma} = e_{\mathcal{A}(X)}\}\\
&=& \{\sigma \in Aut(\Gamma) \ | \ \forall x \in V(X) \ , \ \sigma(X_x) = X_x\}\\
&=& \{\sigma \in Aut(\Gamma) \ | \ \sigma = \Pi_{x \in V(X)}\sigma_x \ s.t. \ \sigma_x \in Sym(V(X_x))\}\\
&=& \prod_{x \in V(X)}Sym(V(X_x)).
\end{eqnarray*}
We use the notation $(x , a)$ for vertices of $X_a$ in $\Gamma$. Without loss
of generality, we can assume that $\sigma_x \in Sym(\{x\} \times
V(X_x))$. For every $f \in \mathcal{A}(X)$, the automorphism $\sigma_f \in
Aut(\Gamma)$ is defined as $\sigma_f((x , a)) = (f(x) , a)$.
Define $D$ to be the set of all such permutations. It is clear
that $D \cong \mathcal{A}(X)$ and
$$D \cap Ker(\varphi) = \{\sigma_f \in D \ | \ \forall x \in V(X) \ , \ f(x) = x\} = \{e_{Aut(\Gamma)}\}.$$
Suppose $\sigma \in Aut(\Gamma)$. Then, there exists $f \in \mathcal{A}(X)$
such that for each $x \in V(X)$, $\sigma(X_x) = X_{f(x)}$. Now,
for every $(x , a) \in V(\Gamma)$,
$$\sigma((x , a)) = (f(x) , \sigma_x(a)) = \sigma_f((x , \sigma_x(a))) = \sigma_f(\sigma_x((x , a))) = (\sigma_fo\sigma_x)((x , a)).$$
Therefore, $\sigma = \sigma_fo\sigma_x \in DKer(\varphi)$ which
shows that $Aut(\Gamma) = DKer(\varphi)$. Since $Ker(\varphi)$
is  normal in $Aut(\Gamma)$,
$$Aut(\Gamma) \cong Ker(\varphi) \rtimes D \cong (\prod_{x \in V(X)}Sym(V(X_x))) \rtimes \mathcal{A}(X).$$
This completes our arguments. \qed

\begin{cor}\label{first cor}

Suppose $\Gamma$ is a reduced complete-empty $X-$join and $X_x$ denotes the
subgraph corresponding to the vertex $x \in V(X)$. Then $Aut(X_x)$ is
isomorphic to a normal subgroup of $Aut(\Gamma)$.

\end{cor}

\begin{proof}
For each $x \in V(X)$, set
$$S_x = \{\sigma \in Aut(\Gamma) \ | \ \forall a \in V(\Gamma) \setminus V(X_x) , \sigma(a) = a\}.$$
It is clear that $S_x \cong Sym(V(X_x))$ and $S_x \leq Aut(\Gamma)$.
Let $\sigma$ and $\theta$ be arbitrary elements of $Aut(\Gamma)$ and $S_x$, respectively.
By the proof of Theorem \ref{main teorem}, for each $y \in V(X)$, there exists $z_y \in V(X)$, such that
$\sigma(X_y) = X_{z_y}$. So
$$\sigma^{-1}\theta\sigma(X_y) = \sigma^{-1}\theta(X_{z_y}) = \sigma^{-1}(X_{z_y}) = X_y,$$
which shows that $S_x$ is isomorphic to a
normal subgroup of $Aut(\Gamma)$.
\end{proof}

\section{Applications}
Suppose $\Gamma$  is a connected graph. If the intersection of each decreasing chain of neighborhoods of vertex
subsets are non-empty then $\Gamma$ is called $NDC$ graph.  For example, every finite graph or infinite graph
in which its vertices have finite degrees satisfies $NDC$ condition.

It is far from true that each graph is $NDC$. To see this, it is enough to check the graph $\Lambda$ with
$V(\Lambda) = \mathbb{R}$ and $E(\Lambda) = \{xy \ | \ |xy| < 1\}$.
For each $i \in \mathbb{N}$, define $A_i = \{i\}$. It is easy to see that
$N(A_i) = (-\frac{1}{i} , \frac{1}{i})$, $A_{i + 1} \subseteq A_i$ and
$\bigcap_{i \in \mathbb{N}}N(A_i) = \emptyset$. Therefore, $\Lambda$
is not an $NDC$ graph.

In this section, we apply our main theorem and its corollary to obtain the main properties of connected $NDC$ graphs.

\begin{thm}\label{qqqqq}

Let $\Gamma$ be a simple connected graph  satisfies $NDC$ condition.
Then $\Gamma$ can be written as a reduced complete-empty $X-$join.

\end{thm}

\begin{proof}

If $\Gamma$ is complete, then $\Gamma$ can be written as a
reduced complete $K_1-$join. So we can assume that $\Gamma$ is
not complete. Define the sets $NN$, $CE$ and $CE_m$ as follows:
\begin{eqnarray*}
NN &=& \{U \subseteq V(\Gamma) \ | \ N(U) \neq \emptyset\},\\
CE &=& \{U \in NN \ | \ \Gamma[U]: \ \text{complete or empty graph}\},\\
CE_m &=& \{U \in CE \ | \ \forall x , y \in U , N(x) \setminus \{y\} =
N(y) \setminus \{x\}\}.
\end{eqnarray*}
The sets $NN$, $CE$ and $CE_m$ are not empty, since they are containing
single points. It is clear that $CE_m$ is a partially ordered set
under set inclusion. We now prove that for each $a \in V(\Gamma)$,
$CE_m$ has a maximal element containing $a$.
Define the chain $\{U_i\}_{i \in I}$, for each non-empty
and arbitrary set $I$, such that one of the members of this chain
is $\{a\}$. We claim that  $\bigcup_{i \in I}U_i \in CE_m$.
If for all $i \in I$, $U_i = \{a\}$, then it is trivial. So assume that $\bigcup_{i \in I}U_i \neq \{a\}$.
Set $\mathfrak{U} = \bigcup_{i \in I}U_i$ and we first show that
$N(\mathfrak{U}) \neq \emptyset$. The following relations prove
that $N(\bigcup_{i \in I}U_i) = \bigcap_{i \in I}N(U_i)$:
$$x \in N(\bigcup_{i \in I}U_i) \Leftrightarrow \forall i \in I, x \in N(U_i) \Leftrightarrow x \in \bigcap_{i \in I}N(U_i).$$
$NDC$ condition implies that $\bigcap_{i \in I}N(U_i) \neq \emptyset$ and so,
$N(\mathfrak{U}) \neq \emptyset$. We now prove that
$\Gamma[\mathfrak{U}]$ is complete or empty. We first assume that
there exists $l \in I$ such that $\Gamma[U_l]$ is complete and non-isomorphic to $K_1$.
Then, for each $i \in I$, such that $U_l \subseteq U_i$, $\Gamma[U_i]$
is a complete graph and so, $\Gamma[\mathfrak{U}]$ is complete.
If for each $i \in I$, $\Gamma[U_l]$ is empty then it is clear that
$\Gamma[\mathfrak{U}]$ is also empty. Therefore, we observe that
in each case $\Gamma[\mathfrak{U}]$ is complete or empty. Choose
elements $x$ and $y$ in $\mathfrak{U}$, then there are $j , k \in
I$ such that $x \in U_j$ and $y \in U_k$. Since $U_i$'s are
chain, without loss of generality, we can assume that $U_j \subseteq
U_k$. Thus $x , y \in U_k$. Since $U_k \in CE_m$, $N(x) \setminus \{y\} =
N(y) \setminus \{x\}$ which proves that $\mathfrak{U} \in CE_m$. It can
easily see that $\mathfrak{U}$ is a maximal member of this chain
and so by Zorn's lemma, $CE_m$ has a maximal element $M$ in the
set $\{ T \in CE_m \ | \ a \in T\}$. The maximal member $M$ is
also the maximal member of $CE_m$, since otherwise, $M$ is
contained in a member of $CE_m$, contradict by maximality of $M$.

The set of all maximal elements of $CE_m$ is denoted by
$\mathcal{M}$. Since $CE_m$ is containing all singletons,
$\bigcup \mathcal{M} = \bigcup CE_m = V(\Gamma)$. To prove the
intersection of the two arbitrary different elements of $\mathcal{M}$ is empty,
by contrary, we assume that
$M_1 , M_2 \in \mathcal{M}$ such that $M_1 \cap M_2 \neq
\emptyset$. It is easy to prove $|M_1| , |M_2| \neq 1$. Now, we consider two cases as follows.

\begin{itemize}

\item[a.]

$|M_1 \cap M_2| \geq 2$. We first notice that $\Gamma[M_1]$
is complete if and only if $ \Gamma[M_1 \cap M_2]$ is complete if and only if
$\Gamma[M_2]$ is complete and in the same way $\Gamma[M_1]$ is empty
if and only if $\Gamma[M_2]$ is an empty graph.
Without loss of generality, we assume that $M_1 \setminus M_2$ is nonempty
and let $y \in M_1 \setminus M_2$.
We now assume that $\Gamma[M_1]$ and $\Gamma[M_2]$
are complete graphs and $x \in M_1 \cap M_2$.
By assumption $N_{\Gamma}(x) \setminus \{y\} = N_{\Gamma}(y) \setminus \{x\}$.
If $M_2 \setminus M_1$
is a nonempty set, then all of its elements are adjacent to $x$ in
$\Gamma$ and we have $M_2 \setminus M_1 \subseteq N_{\Gamma}(y) \setminus \{x\}$.
This implies that $\Gamma[M_1 \cup M_2]$ is complete. Which contradicts
by maximality of $M_1$ and $M_2$ in $CE_m$. Hence $M_2 \setminus M_1 = \emptyset$
and so $M_2 \subsetneq M_1$. This leaded us to another contradiction by
maximality of $M_2$. Next we assume that $\Gamma[M_1]$ and $\Gamma[M_2]$
are both empty graphs. Choose $x \in M_1 \cap M_2$.
Since $N_{\Gamma}(x) \setminus \{y\} = N_{\Gamma}(y) \setminus \{x\}$,
$N_{\Gamma}(x) = N_{\Gamma}(y)$. Moreover, if $z \in M_2 \setminus M_1$,
then $N_{\Gamma}(x) \setminus \{z\} = N_{\Gamma}(z) \setminus \{x\}$
and so $N_{\Gamma}(x) = N_{\Gamma}(z)$. Therefore,
$N_{\Gamma}(x) = N_{\Gamma}(y) = N_{\Gamma}(z)$ which implies that
$M_1 \cup M_2 \in CE_m$, a contradiction. If $M_2 \setminus M_1 = \emptyset$,
then again $M_2 \subsetneq M_1$ and this is all final contradiction.

\item[b.]

$|M_1 \cap M_2| = 1$. If both of $\Gamma[M_1]$ and $\Gamma[M_2]$
are complete graphs or empty graphs, then a similar argument as case (a),
leaded us to a contradiction. Thus, without loss of generality, we can assume that
$\Gamma[M_1]$ is a complete graph and $\Gamma[M_2]$ is empty.
Choose $x \in M_1 \cap M_2$, $y \in M_1 \setminus M_2$ and
$z \in M_2 \setminus M_1$. By definition of $CE_m$,
$N_{\Gamma}(x) \setminus \{z\} = N_{\Gamma}(z) \setminus \{x\}$.
Since $y \in N_{\Gamma}(x) \setminus \{z\}$, $y \in N_{\Gamma}(z) \setminus \{x\}$,
which shows that $z \in N_{\Gamma}(y) \setminus \{x\}$. On the other hand,
$N_{\Gamma}(x) \setminus \{y\} = N_{\Gamma}(y) \setminus \{x\}$,
implies that $z \in N_{\Gamma}(x) \setminus \{y\}$ which is impossible.

\end{itemize}

In each case we lead to a contradiction which shows that $\mathcal{M}$
is a partition of vertices of $\Gamma$. Suppose $M$ and $M^{\prime}$
are arbitrary distinct elements of $\mathcal{M}$, $x \in M$ and $x^{\prime} \in M^{\prime}$.
We will prove that $xx^{\prime} \in E(\Gamma)$ if and only if $\Gamma[M] \sim \Gamma[M^{\prime}]$.
It is enough to assume that $xx^{\prime} \in E(\Gamma)$, $y \in M$ and $y^{\prime} \in M^{\prime}$.
Since $x^{\prime} \in N_{\Gamma}(x)$, $x^{\prime} \in N_{\Gamma}(y)$ and so $x^{\prime}y \in E(\Gamma)$,
which proves that $x^{\prime}$ is adjacent to all elements of $M$.
A similar argument shows that $x$ is adjacent to all elements of $M^{\prime}$.
Since $y \in N_{\Gamma}(x^{\prime}) \setminus \{y^{\prime}\}$ and
$N_{\Gamma}(x^{\prime}) \setminus \{y^{\prime}\} = N_{\Gamma}(y^{\prime}) \setminus \{x^{\prime}\}$,
$y$ and $y^{\prime}$ are adjacent, which implies that all elements of $M$
are adjacent to all elements of $M^{\prime}$. Therefore, $\Gamma[M] \sim \Gamma[M^{\prime}]$
and $xx^{\prime} \not \in E(\Gamma)$ if and only if it is impossible
to find an edge that connects a vertex in $\Gamma[M]$ to another vertex in
$\Gamma[M^{\prime}]$. Define the graph $X$ as follows:
$$V(X) = \mathcal{M} \ \ ; \ \ E(X) = \{M_1M_2 \ | \ \Gamma[M_1] \sim \Gamma[M_2]\}.$$
Suppose $\Gamma[M]$ is the graph corresponding to the vertex $M$ in $X$.
Since any subgraph generated by elements of
$CE_m$ are complete or empty, the subgraph generated by all elements
$\mathcal{M}$ is also complete or empty.
Thus, $\Gamma[M]$ is complete or empty. So, we proved that the graph $\Gamma$
can be written as an $X$-join of complete or empty graphs.
We now prove that this join is reduced. To do this, we consider two arbitrary vertices
$M_1$ and $M_2$ such that $N_{X}(M_1) \setminus \{M_2\} = N_{X}(M_2) \setminus \{M_1\}$.

We first assume that $M_1M_2 \not \in E(X)$ and show that at least one of the graphs
$\Gamma[M_1]$ or $\Gamma[M_2]$ are not empty.
By contrary, suppose both of $\Gamma[M_1]$ and $\Gamma[M_2]$ are empty graphs.
By connectedness of $\Gamma$, there exists
$M_3 \in V(X)$ such that $M_3 \neq M_1 , M_2$ and $M_1M_3 , M_2M_3 \in E(X)$.
It is clear that $M_3 \subseteq N(M_1 \cup M_2)$ and so $N(M_1 \cup M_2) \neq \emptyset$.
Since $M_1$ and $M_2$ are empty, $M_1 \cup M_2$ is empty and hence
$M_1 \cup M_2 \in CE$. Choose arbitrary distinct elements $x_1 , x_2 \in M_1 \cup M_2$.
If $x_1 , x_2 \in M_1$ or $x_1 , x_2 \in M_2$, then one can easily see that
$N_{\Gamma}(x_1) \setminus \{x_2\} = N_{\Gamma}(x_2) \setminus \{x_1\}$,
which shows that $M_1 \cup M_2 \in CE_m$. This is a contradiction by maximality of $M_1$ and $M_2$.
Therefore, without loss of generality, we can assume that $x_1 \in M_1$ and $x_2 \in M_2$.
Obviously, for each $a \in N_{\Gamma}(x_1) \setminus \{x_2\}$, there exists
$M_a \in V(X)$ such that $a \in M_a$ and $M_a \neq M_2$. Since
$ax_1 \in E(\Gamma)$, $M_a \neq M_1$ and so $M_1M_a \in E(X)$.
Hence $M_a \in N_X(M_2) \setminus \{M_1\}$, which implies that $ax_2 \in E(\Gamma)$.
Thus $a \in N_{\Gamma}(x_2) \setminus \{x_1\}$, which shows that
$N_{\Gamma}(x_1) \setminus \{x_2\} \subseteq N_{\Gamma}(x_2) \setminus \{x_1\}$.
By a similar argument as above, we can see that
$N_{\Gamma}(x_2) \setminus \{x_1\} \subseteq N_{\Gamma}(x_1) \setminus \{x_2\}$,
which proves $N_{\Gamma}(x_1) \setminus \{x_2\} = N_{\Gamma}(x_2) \setminus \{x_1\}$.
Therefore, $M_1 \cup M_2 \in CE_m$, which is impossible. This contradiction shows that
at least one of $\Gamma[M_1]$ and $\Gamma[M_2]$ are nonempty.

We now assume that $M_1M_2 \in E(X)$ and prove at least one of the graphs
$\Gamma[M_1]$ and $\Gamma[M_2]$ are not complete. To see this, by contrary,
assume that both of them are complete graphs. If $N_X(M_1) \setminus \{M_2\}$
is empty, then $X \cong K_2$ and graphs corresponding to two vertices of $K_2$
are complete which concludes that $\Gamma$ is complete.
This lead us to a contradiction. Therefore, $N_X(M_1) \setminus \{M_2\} = N_X(M_2) \setminus \{M_1\} \neq \emptyset$,
which shows that there exists $M_4 \in V(X)$ such that
$M_4M_1 , M_4M_2 \in E(X)$. Hence $M_4 \subseteq N(M_1 \cup M_2)$ and so
$N(M_1 \cup M_2) \neq \emptyset$. Since $\Gamma[M_1]$ and $\Gamma[M_2]$
are complete and $M_1M_2 \in E(X)$, $\Gamma[M_1 \cup M_2]$ is a complete graph
and so $M_1 \cup M_2 \in CE$. Choose $y_1 , y_2 \in M_1 \cup M_2$.
If $y_1 , y_2 \in M_1$ or $y_1 , y_2 \in M_2$, then we can easily see that
$N_{\Gamma}(y_1) \setminus \{y_2\} = N_{\Gamma}(y_2) \setminus \{y_1\}$.
So, $M_1 \cup M_2 \in CE_m$ lead us to a contradiction. Then,
without loss of generality, we can assume that $y_1 \in M_1$ and $y_2 \in M_2$.
Choose $a \in N_{\Gamma}(y_1) \setminus \{y_2\}$. If $a \in M_1$ then
$ay_2 \in E(\Gamma)$ and so $a \in N_{\Gamma}(y_2) \setminus \{y_1\}$.
If $a \in M_2$, then $ay_2 \in E(\Gamma)$ which again proves that
$a \in N_{\Gamma}(y_2) \setminus \{y_1\}$. If
$a \not \in M_1 \cup M_2$, then there are $M_3 \in V(X)$,
$M_3 \neq M_1 , M_2$, such that $a \in M_3$. Since $ax \in E(\Gamma)$,
$M_3M_1 \in E(X)$ and so $M_3 \in N_X(M_1) \setminus \{M_2\}$.
This concludes that $M_3 \in N_X(M_2) \setminus \{M_1\}$.
Hence $M_3M_2 \in E(X)$ and $a \in N_{\Gamma}(y_2) \setminus \{y_1\}$.
Therefore, $N_{\Gamma}(y_1) \setminus \{y_2\} \subseteq N_{\Gamma}(y_2) \setminus \{y_1\}$
and similarly $N_{\Gamma}(y_2) \setminus \{y_1\} \subseteq N_{\Gamma}(y_1) \setminus \{y_2\}$.
This leaded us to $M_1 \cup M_2 \in CE_m$ that contradicts by maximality of
$M_1$ and $M_2$. This completes the proof.

\end{proof}

\begin{lem}\label{new}

Assume that $\Gamma$ is an $X-$join graph. Then, there exists
a partition of $V(\Gamma)$ that is denoted by $\mathcal{P}$, such that
$\Gamma / \mathcal{P} \cong X$.

\end{lem}

\begin{proof}

Suppose $X_x$ is a graph corresponding to the vertex $x$ of $X$.
Define $P_x = V(X_x)$ and let $\mathcal{P} = \{P_x \ | \ x \in V(X)\}$.
It is clear that the well-defined function $f : \Gamma / \mathcal{P} \rightarrow X$,
by the criterion $f(P_x) = x$, is a bijection. So, it is enough to prove that
$f$ is a graph isomorphism. Following discussion shows that $P_xP_y \in E(\Gamma / \mathcal{P})$
if and only if $xy \in E(X)$, for every $x , y \in V(X)$:
\begin{eqnarray*}
P_xP_y \in E(\Gamma / \mathcal{P}) &\Leftrightarrow& \exists p_x \in P_x , p_y \in P_y \ s.t. \ p_xp_y \in E(\Gamma)\\
&\Leftrightarrow& \exists p_x \in V(X_x) , p_y \in V(X_y) \ s.t. \ p_xp_y \in E(\Gamma)\\
&\Leftrightarrow& xy \in E(X).
\end{eqnarray*}

\end{proof}

In the following theorem, it is proved that the graph $X$ in Theorem
\ref{qqqqq} is unique.

\begin{thm}\label{unique}

Let $\Gamma$ be a reduced complete-empty $X-$ and $Y-$join. Then $X \cong Y$.

\end{thm}

\begin{proof}

By Lemma \ref{new}, $V(\Gamma)$ have two partitions ${\mathcal P}_1$
and ${\mathcal P}_2$ such that $\Gamma / {\mathcal P}_1 \cong X$
and $\Gamma / {\mathcal P}_2 \cong Y$. If ${\mathcal P}_1 = {\mathcal P}_2$
then clearly, $X \cong Y$. Suppose ${\mathcal P}_1 \neq {\mathcal P}_2$.
Without loss of generality, assume that there exists
$P_1 \in {\mathcal P}_1 \setminus {\mathcal P}_2$.
Then there exist $P_2 , P_2^{\prime} \in {\mathcal P}_2$,
such that $P_1 \cap P_2 , P_1 \cap P_2^{\prime} \neq \emptyset$.
Let $y_2 , y_2^{\prime} \in V(Y)$ which are corresponding to the graphs
$\Gamma[P_2]$ and $\Gamma[P_2^{\prime}]$, respectively.
We consider $y \in N_Y(y_2) \setminus \{y_2^{\prime}\}$
and denote the graph corresponding to $y$ by $Y_y$.
So, there exists $P_2^{\prime \prime} \in {\mathcal P}_2$
such that $Y_y = \Gamma[P_2^{\prime \prime}]$.
We will consider two separate cases as follows:

\begin{enumerate}

\item[a.]

\textit{$P_2^{\prime \prime} \setminus P_1 \neq \emptyset$}.
Suppose $a \in P_2^{\prime \prime} \setminus P_1$.
Hence there exists a part $P_1^{\prime} \in {\mathcal P}_1$ such that
$a \in P_1^{\prime}$. Since $a$
is adjacent to all vertices of $\Gamma[P_2]$ and $P_1 \cap P_2 \neq \emptyset$,
all elements of $P_1$ are adjacent to all elements of $P_1^{\prime}$.
Thus, $a$ is adjacent to all elements of $P_1 \cap P_2^{\prime}$
and so $\Gamma[P_2^{\prime \prime}] \sim \Gamma[P_2^{\prime}]$.
Hence, $y \in N_Y(y_2^{\prime}) \setminus \{y_2\}$ and
$N_Y(y_2) \setminus \{y_2^{\prime}\} \subseteq N_Y(y_2^{\prime}) \setminus \{y_2\}$, as desired.

\item[b.]

\textit{$P_2^{\prime \prime} \subseteq P_1$}.
Since $\Gamma[P_2] \sim \Gamma[P_2^{\prime \prime}]$,
all elements of $P_2^{\prime \prime}$ are adjacent to all elements of $P_1 \cap P_2^{\prime}$.
Hence, $E(\Gamma[P_1]) \neq \emptyset$ and $\Gamma[P_1]$ is complete.
This concludes that all elements of
$P_2^{\prime \prime}$ are adjacent to all elements of $P_1 \cap P_2^{\prime}$
and so, are adjacent to vertices of $\Gamma[P_2^{\prime}]$.
Now, one can prove that
$y \in N_Y(y_2^{\prime}) \setminus \{y_2\}$.
Therefore, $N_Y(y_2) \setminus \{y_2^{\prime}\} \subseteq N_Y(y_2^{\prime}) \setminus \{y_2\}$.

\end{enumerate}

Now, using a similar argument,
shows that
$N_Y(y_2^{\prime}) \setminus \{y_2\} \subseteq N_Y(y_2) \setminus \{y_2^{\prime}\}$ and so
$N_Y(y_2^{\prime}) \setminus \{y_2\} = N_Y(y_2) \setminus \{y_2^{\prime}\}$.

The inclusion relation between $P_2$ and $P_1$ and between
$P_2^{\prime}$ and $P_1$, is discussed in the following two cases.

\begin{enumerate}

\item[1.]

\textit{$P_2, P_2^{\prime} \subseteq P_1$}.
If $\Gamma[P_1]$ is complete,
then $\Gamma[P_2]$ and $\Gamma[P_2^{\prime}]$ are complete graphs
in which $\Gamma[P_2] \sim \Gamma[P_2^{\prime}]$.
So, $y_2y_2^{\prime} \in E(Y)$ which contradicts by reducibility of $\Gamma$ over $Y$.
If $\Gamma[P_1]$ is an empty graph,
then $\Gamma[P_2]$ and $\Gamma[P_2^{\prime}]$ are empty
and so $\Gamma[P_2] \not \sim \Gamma[P_2^{\prime}]$.
Thus, $y_2y_2^{\prime} \not \in E(Y)$, which is again a
contradiction due to reducibility of $\Gamma$ over $Y$.

\item[2.]

\textit{$P_2 \cup P_2^{\prime} \not \subseteq P_1$}.
Without loss of generality, assume that $P_2 \not \subseteq P_1$
and $P_1 \cap P_2 \neq \emptyset$.
Let $a \in P_1 \cap P_2$, $b \in P_1 \cap P_2^{\prime}$ and $c \in P_2 \setminus P_1$.
There is $P_1^{\prime} \in {\mathcal P}_1$ such that $c \in P_1^{\prime}$.
It is easy to see that $\Gamma[P_1]$ is complete, if and only if $ab \in E(\Gamma)$.
A similar argument shows that $\Gamma[P_2]$ is a complete graph, if and only if $ac \in E(\Gamma)$.
Hence,
\begin{eqnarray*}
ab \in E(\Gamma) &\Leftrightarrow& \Gamma[P_2] \sim \Gamma[P_2^{\prime}]\\
&\Leftrightarrow& bc \in E(\Gamma)\\
&\Leftrightarrow& \Gamma[P_1] \sim \Gamma[P_1^{\prime}]\\
&\Leftrightarrow& ac \in E(\Gamma),
\end{eqnarray*}
which implies that $\Gamma[P_1]$ is complete, if and only if $\Gamma[P_2]$ is a
complete graph. It is easy to prove that $\Gamma[P_2] \sim \Gamma[P_2^{\prime}]$,
if and only if $y_2y_2^{\prime} \in E(Y)$. If $P_2^{\prime} \subseteq P_1$, then
$\Gamma[P_1]$ is a complete graph, if and only if $\Gamma[P_2^{\prime}]$
is complete, which contradicts by reducibility of $\Gamma$ over $Y$.
Also, if $P_2^{\prime} \not \subseteq P_1$ and $P_1 \cap P_2^{\prime} \neq \emptyset$,
then a similar argument as $P_2$, shows that $\Gamma[P_1]$ is complete, if and only if
$\Gamma[P_2^{\prime}]$ is a complete graph, which is an another contradiction.

\end{enumerate}

Now, if there exists an element in ${\mathcal P}_2$ contains $P_1$,
then the changing role of $X$ and $Y$, lead us to a contradiction.
This completes proof.

\end{proof}

By the last theorem, if $\Gamma$ can be written as a reduced complete-empty $X-$join, we call
$X$ to be a characteristic graph of $\Gamma$ and denote it by $\chi(\Gamma) = X$.
The following corollary is an immediate consequence of the
previous theorem.

\begin{cor}\label{second cor}
Suppose $\Gamma$ is a connected, $NDC$ graph and
$Aut(\Gamma)$ is simple. Then $\Gamma \cong \chi(\Gamma)$.
\end{cor}

\begin{proof}

Theorems \ref{qqqqq} and
\ref{unique} conclude that $\Gamma$ can be written as a reduced complete-empty $X-$join
of some graphs for a unique graph $X$, up to isomorphism. If we denote $X_x$ as a graph
corresponding to the vertex $x \in V(X)$, then
By Corollary \ref{first cor},
$Aut(X_x)$ is isomorphic to a normal
subgroup of $Aut(\Gamma)$. Since $Aut(X_x) \cong Sym(V(X_x))$ and
$Aut(\Gamma)$ is simple, $|V(X_x)| = 1$ which shows that $\Gamma = X = \chi(\Gamma)$.
\end{proof}


\section{Concluding Remark}

In this paper, reduced complete-empty $X-$join of graphs together with their automorphism group are studied. It is proved that the automorphism group of such graphs is a semi-direct product of two groups. Our calculations with graphs of small orders suggest the following open question:

\begin{qu}
Is it true that every simple connected graph can be written as a  reduced complete-empty $X-$join of some graphs?
\end{qu}

It is well-known that most of graphs have trivial automorphism group. If the above question has an affirmative answer then we can immediately prove that the most of graphs have trivial $X_x$, $x \in V(X)$.

\vskip 3mm

\noindent{\bf Acknowledgement.} The research of the authors
are partially supported by the University of Kashan under grant no
364988/67.


\end{document}